\documentclass{amsart}

\newtheorem{theorem}{Theorem}[section]
\newtheorem{lemma}[theorem]{Lemma}
\newtheorem{proposition}[theorem]{Proposition}

\theoremstyle{definition}
\newtheorem{definition}[theorem]{Definition}

\theoremstyle{remark}
\newtheorem{remark}[theorem]{Remark}

\numberwithin{equation}{section}

\begin{document}

\title[Stability]{Balanced metrics and Chow stability of projective bundles over Riemann Surfaces}%
\author{Reza SEyyedali}%
\address{Mathematics Department, University of California, Irvine}%
\email{rseyyeda@math.uci.edu}%




\begin{abstract}

In 1980, I. Morrison proved that slope stability of a vector
bundle of rank $2$ over a compact Riemann surface implies Chow
stability of the projectivization  of the bundle with respect to
certain polarizations. We generalized Morrison's result to higher rank vector bundles over compact
algebraic manifolds of arbitrary dimension that admit constant
scalar curvature metric and have discrete automorphism group. In this article, we give a simple proof for  polarizations  $\mathcal{O}_{\mathbb{P}E^*}(d)\otimes \pi^* L^k$, where $d$ is a positive integer, $k \gg 0$ and the base manifold is a compact Riemann surface of genus $g \geq 2$.

\end{abstract}

\maketitle

\section{Introduction}

In \cite{M}, Morrison proved that for the
projectivization of a rank two holomorphic vector bundle over a
compact Riemann surface, Chow stability is equivalent to the
stability of the bundle. In \cite{S}, We generalized one direction of
Morrison's result for higher rank vector bundles over compact
algebraic manifolds of arbitrary dimension that admit constant
scalar curvature metric and have discrete automorphism group (\cite{S}).

Let $X$ be a compact complex manifold
of dimension $m$ and $\pi : E\rightarrow X$ be a holomorphic vector
bundle of rank $r$ with dual bundle $E^*$. This gives a
holomorphic fibre bundle $\mathbb{P}E^*$ over $X$ with fibre
$\mathbb{P}^{r-1}$. One can pull back the vector bundle $E$ to
$\mathbb{P}E^*$. We denote the tautological line bundle on
$\mathbb{P}E^*$ by $\mathcal{O}_{\mathbb{P}E^*}(-1)$ and its dual
by $\mathcal{O}_{\mathbb{P}E^*}(1)$. Let $L \rightarrow X$ be an
ample line bundle on $X$ and $\omega \in 2\pi  c_{1}(L)$ be a
K\"ahler form. Since $L$ is ample, there is an integer $k_{0} $ so
that for any $k\geq k_{0}$, $\mathcal{O}_{\mathbb{P}E_{k}^*}(1)$
is very ample over $\mathbb{P}E_{k}^*$, where $E_{k}=E\otimes
L^{\otimes k}$. Note that there is a canonical isomorphism $\mathbb{P}E_{k}^* \cong \mathbb{P}E^*$
and $\mathcal{O}_{\mathbb{P}E_{k}^*}(1)\cong
\mathcal{O}_{\mathbb{P}E^*}(1) \otimes \pi^* L^{k}$. The main theorem of \cite{S}
is the following:

\begin{theorem}\label{thm1}
Suppose that $Aut(X)$ is discrete and $X$ admits a constant scalar
curvature K\"ahler metric in the class of $2\pi  c_{1}(L)$. If $E$
is Mumford stable, then 
$$(\mathbb{P}E^*,\mathcal{O}_{\mathbb{P}E^*}(1)\otimes \pi^* L^k)$$
is Chow stable for $k \gg k_{0}$.

\end{theorem}

One of the earliest results in this spirit is the work of Burns
and De Bartolomeis in \cite{BD}. They constructed a ruled surface
which does not admit any extremal metric in certain cohomology
class. In \cite{H1}, Hong proved that there are constant scalar
curvature K\"ahler metrics on the projectivization of stable
bundles over curves. In \cite{H2} and \cite{H3}, he generalized
this result to higher dimensions with some extra assumptions.
Combining Hong's results with Donaldson's,
$(\mathbb{P}E^*,\mathcal{O}_{\mathbb{P}E^*}(n) \otimes \pi^*L^m)$ is Chow
stable for $m,n \gg 0$ when the bundle $E$ is stable. Note that it
differs from our result, since it implies the Chow stability of
$(\mathbb{P}E^*,\mathcal{O}_{\mathbb{P}E_{m}^*}(n))$ for $n, m$ big
enough.

In \cite{RT}, Ross and Thomas developed the notion of slope
stability for polarized algebraic manifolds. As one of the
applications of their theory, they proved that if
$(\mathbb{P}E^*,\mathcal{O}_{\mathbb{P}E^*}(1)\otimes \pi^* L^k)$
is slope semi-stable for $k \gg 0$, then $E$ is a slope semistable
bundle and $(X,L)$ is a slope semistable manifold.Again note that
they look at stability of $\mathbb{P}E^*$ with respect to
polarizations $\mathcal{O}_{\mathbb{P}E_{m}^*}(n)$ for $n$ big
enough. For the case of one dimensional base, however they showed
stronger results. In this case they proved that if
$(\mathbb{P}E^*,\mathcal{O}_{\mathbb{P}E^*}(1)\otimes \pi^* L)$ is
slope (semi, poly) stable for any ample line bundle $L$, then $E$
is a slope (semi, poly) stable bundle.

In order to prove Theorem \ref{thm1} we used the concept of
balanced metrics. Combining the results of Luo, Phong, Sturm, Wang
and Zhang on the relation between balanced metrics and stability, we proved the following.

\begin{theorem}(\cite{S})\label{thm2}
Suppose that $Aut(X)$ is discrete and $X$ admits a constant scalar
curvature K\"ahler metric in the class of $2\pi  c_{1}(L)$. If $E$
is Mumford stable, then 
$$(\mathbb{P}E^*,\mathcal{O}_{\mathbb{P}E^*}(1)\otimes \pi^* L^k)$$
admits balanced metrics for $k \gg 0$.

\end{theorem}

In this paper, we give another proof of Theorem \ref{thm1} in the case of one dimensional base $X$. The proof is simple in this case and can be generalized to polarizations $\mathcal{O}_{\mathbb{P}E^*}(d)\otimes \pi^* L^k$ for any positive integer $d$ and $k \gg 0$. The main theorem of this paper is the following.

\begin{theorem}\label{thm2b}

Let $X$ be a compact Riemann surface of genus $g \geq 2$ and
$E\rightarrow X$ be a holomorphic vector bundle on $X$. Let $d$ be a positive integer. If $E$ is Mumford stable, then
$$(\mathbb{P}E^*,\mathcal{O}_{\mathbb{P}E^*}(d)\otimes \pi^* L^k)$$ admits balanced
metrics $g_{k}$ for $k \gg 0$.

\end{theorem}

The Hitchin-Kobayashi correspondence implies that the stable bundle $E$ admits a Hermitian-Einstein metric $h_{\infty}$. A simple calculation shows that the Hermitian metric $\textrm{Sym}^d h_{\infty}$ on $\textrm{Sym}^d E$ is Hermitian-Einstein. Therefore the vector bundle $\textrm{Sym}^d E$ is stable. By a theorem of Wang , we know that there exist balanced metrics $H^{(k)}$ on $\textrm{Sym}^d E \otimes L^k$. This
means that there exists a basis $s_{1},...,s_{N}$ for $H^0(X,\textrm{Sym}^d E\otimes
L^k)$ such that
$$\sum s_{i} \otimes s_{i}^{*_{H^{(k)}}}= I_{\textrm{Sym}^d E}$$ $$\int_{X} \langle s_{i},
s_{j}\rangle_{H^{(k)}} \omega=\frac{rV}{N} \delta_{ij}.$$

Using the canonical isomorphism between $H^0(X,\textrm{Sym}^d E \otimes L^{k})$ and
$H^0(\mathbb{P}E^*,\mathcal{O}_{\mathbb{P}E^*}(d) \otimes L^{k})$, we get a
sequence of Hermitian metrics $\widehat{H^{(k)}}$ on
$\mathcal{O}_{\mathbb{P}E^*}(d)\otimes L^k$. We prove that the sequence 
$\widehat{H^{(k)}}$ is "almost balanced", i.e.  $$\int
\langle \widehat{s_{i}},
\widehat{s_{j}}\rangle_{\widehat{H^{(k)}}}
dvol_{\widehat{h^{(k)}}}= D^{(k)}\delta_{ij}+M^{(k)}_{ij},$$ where $D^{(k)}\rightarrow
C_{r,d}$ as $k \rightarrow \infty$ (see \eqref{eq11} for the definition of $C_{r,d}$) and $M^{(k)}$ is a trace-free
Hermitian matrix such that $||M^{(k)}||_{op}=o(k^{-\infty})$  as
$k\rightarrow \infty$ .

The next step is to perturb these almost balanced metrics to get
balanced metrics. As pointed out by Donaldson, the problem of
finding balanced metric can be viewed also as a finite dimensional
moment map problem solving the equation $M^{(k)}=0$. Indeed,
Donaldson shows that $M^{(k)}$ is the value of a moment map
$\mu_{D}$ on the space of ordered bases with the obvious action of
$SU(N)$. Now, the problem is to show that if for some ordered
basis $\underline{s}$, the value of moment map is very small, then
we can find a basis at which moment map is zero. The standard
technique is flowing down $\underline{s}$ under the gradient flow
of $|\mu_{D}|^2$ to reach a zero of $\mu_{D}$. We need a
Lojasiewicz type inequality to guarantee that the flow converges
to a zero of the moment map. This was done in \cite {S} by adapting
Phong-Sturm proof to our situation. 
In \cite {S2}, we generalize Theorem \ref{thm2b} to higher dimensional base manifolds that admits cscK metrics and do not have any nonzero holomorphic vector fields. 
After this work was completed, we became aware of the preprint \cite{DZ}.

{\textbf{Acknowledgements:} I am sincerely grateful to Richard Wentworth for many helpful discussions and suggestions
and his continuous help,
support and encouragement.}

\section{Preliminaries}

Let $V$ be a Hermitian vector
space of dimension $r$. The projective space $\mathbb{P}V^*$ can
be identified with the space of hyperplanes in $V$ via $$f\in V-\{0\}
\rightarrow {ker(f)}=V_{f}\subset V.$$  There is a natural isomorphism
between $V$ and
$H^{0}(\mathbb{P}V^*,\mathcal{O}_{\mathbb{P}V^*}(1))$ which sends
$v\in V$ to $\hat{v}\in
H^{0}(\mathbb{P}V^*,\mathcal{O}_{\mathbb{P}V^*}(1)) $ such that
for any $f \in V^*, \hat{v}(f)=f(v) $. Any Hermitian inner product $h$ on $V$ induces a Hermitian inner product  $\widehat{h}$ on $\mathcal{O}_{\mathbb{P}V^*}(1)$ as follows:

 $$\langle \hat{v},\hat{w}\rangle _{\widehat{h}}(f)=
\frac{f(v)\overline{f(w)}}{|f|^{2}},$$ for $v,w \in V$ and $f
\in V^*$.

For any positive integer $d$, define an equivalence relation $\sim$ on $V^{\otimes d}$ by $$v_{1} \otimes \dots \otimes v_{d}\sim v_{\sigma(1)} \otimes \dots \otimes v_{\sigma(d)}, \,\,\,\,\,\,\,\,\, \sigma \in S_{d}.$$  We define $\displaystyle \textrm{Sym}^d V =V^{\otimes d}/\sim$ and simply denote the class of $v_{1} \otimes \dots \otimes v_{d}$ in $\textrm{Sym}^d V$ by $v_{1}  \dots  v_{d}.$ Similar to the case of $d=1$ any Hermitian inner product $h$ on $V$ induces a Hermitian inner product  $\textrm{Sym}^d h$ on $\textrm{Sym}^d V$ by 
$$\langle v_{1}  \dots  v_{d}, w_{1}  \dots  w_{d}\rangle_{\textrm{Sym}^d h}= \frac{1}{d!}\sum _{\sigma \in S_{d}} \langle v_{1},w_{\sigma(1)}\rangle_{h} \dots \langle v_{d},w_{\sigma(d)}\rangle_{h} .$$  

\begin{remark}

Let $e_{1}, \dots , e_{r}$ be an orthonormal basis for $V$ with respect to $h$, then the set
 $$\{\Big( \frac{i_{1}! \dots i_{r}!}{d!} \Big)^{\frac{-1}{2}} e_{1}^{i_{1}} \dots e_{r}^{i_{r}} |\,\,\, 0\leq i_{\alpha} \leq d,   \,\,\,\sum_{\alpha=1}^r i_{\alpha}=d \}$$  forms an orthonormal basis for $\textrm{Sym}^d V$ with respect to $\textrm{Sym}^d h$.
 
\end{remark}

There is a natural isomorphism between $\textrm{Sym}^d V$ and
$H^{0}(\mathbb{P}V^*,\mathcal{O}_{\mathbb{P}V^*}(d))$ which sends
$v_{1}\dots v_{d}\in \textrm{Sym}^d V$ to $\widehat{v_{1}\dots v_{d} }\in
H^{0}(\mathbb{P}V^*,\mathcal{O}_{\mathbb{P}V^*}(d)) $ defined by 
\begin{equation}\label{eq1n}\widehat{v_{1}\dots v_{d} }([v^*])(w_{1}^* \otimes \dots \otimes w_{d}^*)= w_{1}^*(v_{1})\dots w_{d}^*(v_{d}),\end{equation} where $v^*\in V^*-\{0\}$ and $w_{i}^* \in V^*$ are scalar multiple of $v^*$. There exist complex numbers $\lambda_{1}, \dots ,\lambda_{d}$ such that $w_{i}^*=\lambda_{i} v^*. $ Thus, 
$$\widehat{v_{1}\dots v_{d} }([v^*])(w_{1}^* \otimes \dots \otimes w_{d}^*)=\lambda_{1} \dots \lambda_{d} v^*(v_{1})\dots v^*(v_{d})$$ and therefore  \eqref{eq1n} defines a well-defined section of $\mathcal{O}_{\mathbb{P}V^*}(d) $.

For any Hermitian inner product $H$ on $\textrm{Sym}^d V$, we define a metric $\hat{H}$ on $\mathcal{O}_{\mathbb{P}V^*}(d) $ by 
$$\langle \hat{s},\hat{t}\rangle_{\hat{H}}[v]= \frac{v^{\otimes d}(s) \overline{v^{\otimes d}(t)}}{|v \dots v|_{H}^{2}}.$$
In particular, we have $$\langle\hat{s},\hat{t}\rangle_{\widehat{\textrm{Sym}^d h}}[v]= \frac{v^{\otimes d}(s) \overline{v^{\otimes d}(t)}}{|v |_{h}^{2d}}.$$
The following lemmas are straight forward.

\begin{lemma}\label{lem1}

For any Hermitian inner product $h$ on $V$, we have  $$\hat{h}^{\otimes d}=\widehat{\textrm{Sym}^d h}.$$
\end{lemma}

\begin{lemma}\label{lem2}
 There exists a constant $C_{r,d}$ such that for any $v,w \in \textrm{Sym}^d V$ and any Hermitian inner product $h$ on $V$,
 
 \begin{equation}\label{eq11} d^{r-1}\int_{\mathbb{P}V^*} \langle \hat{v},\hat{w}\rangle_{\widehat{\textrm{Sym}^d h}}
\frac{\omega_{\textrm{FS},h}^{r-1}}{(r-1)!}= C_{r,d} \langle v,w\rangle_{\textrm{Sym}^d h},\end{equation}
where $\omega_{\textrm{FS},h}=i\bar{\partial}\partial \log \hat{h}$.

\end{lemma}

\begin{remark}

Let $H$ be a Hermitian inner product on $V$. Suppose there exists a constant $C$ such that $$\int_{\mathbb{P}V^*} <\hat{v},\hat{w}>_{\widehat{H}}
\frac{\omega_{\textrm{FS},H}^{r-1}}{(r-1)!}= C <v,w>_{H},$$ for any $v,w \in \textrm{Sym}^d V$. Then $H=\textrm{Sym}^d h$ for some Hermitian inner product $h$ on $V$.

\end{remark}

\begin{lemma}\label{lem7}

Let $h_{0}$ and $h$ be Hermitian inner products on $V$. If
$||h-h_{0}||_{h_{0}}\leq \epsilon$, then
$||\hat{h}-\hat{h}_{0}||_{C^2(\hat{h}_{0})}\leq C \epsilon,$ for a constant $C$ depends only on $r$ and $d$.

\end{lemma}

\begin{lemma}\label{lem8}

Let $X$ be a K\"ahler manifold of dimension $n$ and $\Omega_{0} $
and $\Omega$ be two K\"ahler forms on $X$. There exists a constant $C$ depends only on the dimension of $X$ such that if $||\Omega
-\Omega_{0}||_{C^0(\Omega_{0})}\leq \epsilon$, then
$\displaystyle \Big|\frac{\Omega^n-\Omega_{0}^n}{\Omega_{0}^n} \Big|\leq C \epsilon$. 

\end{lemma}

\begin{proposition}\label{prop1}

Let $h$ be a Hermitian inner product on $V$ and $H$ be a Hermitian inner product on  $\textrm{Sym}^d V$ such that $||H-\textrm{Sym}^d h||_{\textrm{Sym}^d h} < \min(\epsilon, \frac{1}{2})$.
 Then for any $v,w \in \textrm{Sym}^d V$, we have

$$\Big|d^{r-1}\int_{\mathbb{P}V^*} \langle\widehat{v},\widehat{w}\rangle_{\widehat{H}}
\frac{\omega_{FS,h}^{r-1}}{(r-1)!}-C_{r,d}\langle   v,w
\rangle_{H} \Big| \leq C\epsilon |v|_{H}|w|_{H},$$ where $C$ is a constant depends only on $r$ and $d$.
\end{proposition}

\begin{proof}
Let $e_{1}, \dots e_{K}$ be an orthonormal basis for $\textrm{Sym}^d V$ with respect to $\textrm{Sym}^d h$. Define $H_{ij}=H(e_{i},e_{j})$ and $\epsilon_{ij}=H_{ij}-\delta_{ij}$.  We have $|\epsilon_{ij}|=|H_{ij}-\delta_{ij} | \leq \epsilon$. Given $v,w \in \textrm{Sym}^d V$, we can write $v= \sum a_{i}e_{i}$ and $w= \sum b_{i}e_{i}$. We have 

\begin{align*} \Big| \langle v,w\rangle_{\textrm{Sym}^d h} - \langle v,w \rangle_{H} \Big| &=\Big|  \sum a_{i} \bar{b_{i}}-     \sum a_{i} \bar{b_{j}}H_{ij}       \Big|=\Big|  \sum a_{i} \bar{b_{j}}\epsilon_{ij}       \Big|\\& \leq \epsilon \sum |a_{i}| |\bar{b_{j}}| \leq K \epsilon \sum |a_{i}|^2 \sum |b_{j}|^2\\&= \epsilon K |v|_{\textrm{Sym}^d h}|w|_{\textrm{Sym}^d h}. \end{align*}
The last inequality follows from Cauchy-Shwartz inequality. 
By a unitary change of basis we may assume that $H_{ij}=0$ if $i \neq j$. Therefore the basis $\{f_{1}= H_{11}^{\frac{-1}{2}}e_{1}, \dots f_{K}=H_{11}^{\frac{-1}{2}}e_{K}\}$ is an orthonormal basis for $\textrm{Sym}^d V$ with respect to $H$. We have $\frac{1}{2}\leq H_{ii} \leq \frac{3}{2}$ since $|H_{ii}-1| \leq \frac{1}{2}$. Thus, $$\Big ||f_{i}|^2_{\textrm{Sym}^d h}-|f_{i}|^2_{H}\Big |=|1-H_{ii}^{-1}|=\frac{|1-H_{ii}|}{H_{ii}}\leq 2\epsilon.$$
Therefore by the same argument, we conclude that \begin{equation}\label{eq2n} \Big| \langle v,w\rangle_{\textrm{Sym}^d h} -\langle v,w \rangle_{H} \Big| \leq 2 \epsilon K |v|_{H}|w|_{H}. \end{equation} Applying \eqref{eq2n}, Lemma \ref{lem2}, Lemma \ref{lem7} and Lemma \ref{lem8}, we have
\begin{align*}&\Big|d^{r-1}\int_{\mathbb{P}V^*} \langle\widehat{v},\widehat{w}\rangle_{\widehat{H}}\frac{\omega_{FS,h}^{r-1}}{(r-1)!}-C_{r, d} \langle v,w \rangle _{H} \Big|\\ \,\,\,&\leq \Big|d^{r-1}\int_{\mathbb{P}V^*}\langle\widehat{v},\widehat{w}\rangle_{\widehat{H}}\frac{\omega_{FS,h}^{r-1}}{(r-1)!}-d^{r-1}\int_{\mathbb{P}V^*}\langle\widehat{v},\widehat{w}\rangle_{\widehat{\textrm{Sym}^d h}}
\frac{\omega_{FS,h}^{r-1}}{(r-1)!}\Big|+ C_{r, d} \Big| \langle v,w \rangle_{\textrm{Sym}^d h} - \langle v,w \rangle_{H} \Big|\\&
\leq d^{r-1} \int_{\mathbb{P}V^*}\Big|\langle\widehat{v},\widehat{w}\rangle_{\widehat{H}} -\langle\widehat{v},\widehat{w}\rangle_{\hat{\textrm{Sym}^d h}}\Big|
\frac{\omega_{FS,h}^{r-1}}{(r-1)!}+ C_{r, d} \Big| \langle v,w \rangle_{\textrm{Sym}^d h} -\langle v,w \rangle_{H} \Big| 
\\&\leq
C\epsilon d^{r-1}\int_{\mathbb{P}V^*}|\widehat{v}|_{\widehat{\textrm{Sym}^d h}} |\widehat{w}|_{\widehat{\textrm{Sym}^d h}}     \frac{\omega_{FS,h}^{r-1}}{(r-1)!} +2K C_{r,d} \epsilon |v|_{H}|w|_{H}\\&\leq\epsilon(Cd^{r-1} V+2K C_{r,d})|v|_{H}|w|_{H}\end{align*}
The last inequality follows from the fact that $\sup_{\mathbb{P}V^*}|\widehat{v}|_{\widehat{\textrm{Sym}^d h}}=|v|_{\textrm{Sym}^d h} $.

\end{proof}

\section{ Balanced Metrics On Holomorphic Vector Bundles}

Let $(X,\omega_{0})$ be a compact K\"ahler manifold of dimension
$n$ and $(L,g)$ be an ample holomorphic Hermitian line bundle over
$X$ such that $ i\bar{\partial}\partial\log g=\omega_{0}.$ Let $E$ be a holomorphic vector
bundle of rank $r $ over $X$. By possibly tensoring with high power of the ample line bundle $L$, we may assume that
$E$ is very ample. Therefore we can embed $X$ into $G(r, H^{0}(X,E)^*)$, the Grassmanian of $r$-planes in $ H^{0}(X,E)^*$.
Indeed, for any $x \in X$, we have the evaluation map
$H^{0}(X,E)\rightarrow E_{x}$, which sends $s$ to $s(x)$. Since
$E$ is globally generated, this map is a surjection. So its dual
is an inclusion of $E_{x}^* \hookrightarrow H^{0}(X,E)^*$, which
determines an $r$-dimensional subspace of $H^{0}(X,E)^*$.
Therefore we get a map $\iota: X \rightarrow G(r, H^{0}(X,E)^*)$.
Since $E$ is very ample, $\iota$ is an embedding. Clearly we have
$\iota^*U_{r}=E^*$, where $U_{r}$ is the tautological vector
bundle on $G(r, H^{0}(X,E)^*)$, i.e. at any $r$-plane in $G(r,
H^{0}(X,E)^*)$, the fibre of $U_{r}$ is exactly that $r$-plane. A
choice of basis for $H^{0}(X,E)$ gives an isomorphism between
$G(r, H^{0}(X,E)^*)$ and the standard Grassmanian $G(r,N)$, where $N= \dim
H^{0}(X,E)$. We have the standard Fubini-Study Hermitian metric on
$U_{r}$, so we can pull it back to $E$ and get a Hermitian metric
on $E$.

\begin{definition}\label{def1n}

The embedding is called balanced if $$\int_{X} \, \langle s_{i},
s_{j} \rangle_{\iota^* h_{\textrm{FS}}} \, \frac{\omega^{n}}{n!}=C
\delta_{ij}.$$ Notice that being balanced depends on the choice of
the K\"ahler form. A Hermitian metric on $E$ is called balanced (more precisely $\omega$-balanced) if it is the pull back $\iota^* h_{\textrm{FS}}$, where   $\iota$ is a balanced embedding.

\end{definition}

Equivalently, we can formulate the definition of balance metrics in terms of Bergman kernels.

\begin{definition}
Let $h$ be a Hermitian metric on $E$ and  $s_{1},...,s_{N}$ be an orthonmal basis for $H^{0}(X,E)$ with respect to the inner product $$\langle s,t\rangle=\int_{X} \langle s(x),t(x)\rangle_{h}\frac{\omega_{0}^n}{n!}$$
The Bergman kernel of $(E,h)$ is an endomorhism of $E$ defined by $$B(h, \omega_{0})=\sum_{i=1}^N s_{i}
\otimes s_{i}^{*_{h}}.$$ Note that $B(h, \omega_{0})$ does
not depend on the choice of the orthonmal basis. 

A Hermitian metric $h$ on $E$ is balanced if and only if $B(h, \omega_{0})= C I_{E}$ for a positive constant $C$.

\end{definition}

We recall Catlin-Tian-Yau-Zeldich asymptotic expansion of Bergman kernel.

\begin{theorem}(\cite{C}, \cite{Z})
Let $(X,\omega_{0})$ be a compact K\"ahler manifold of dimension
$n$ and $(L,g)$ be an ample holomorphic Hermitian line bundle over
$X$ such that $ i\bar{\partial}\partial\log g=\omega_{0}.$ For any Hermitian metric $h$ on the vector bundle $E$, there exist smooth endomorphisms 
$A_{i}(h) \in \Gamma(X, End(E))$ such  that the
following  asymptotic expansion holds as $k \rightarrow \infty$
\begin{equation}\label{eq1}B(h \otimes g^{\otimes k}, \omega_{0}) \sim k^n+A_{1}(h)k^{n-1}+\dots .\end{equation}

\end{theorem}

There is a close relationship between stability of vector bundles and the existence of balanced
metrics given by the following theorem of Wang.

\begin{theorem}(\cite{W},\cite[Theorem 1.2]{W2})\label{thm5}
The bundle $E$ is Gieseker stable if and only if there exist
balanced metrics $h^{(k)}$ on $E \otimes L^k$ for $k \gg 0$. In
addition if there exists a Hermitian metric $h_{\infty}$ on $E$
such that $h_{k}\rightarrow h_{\infty}$ in $C^{\infty}$, then \begin{equation}\label{eq10}\frac{i}{2\pi}
\Lambda F_{(E,h_{\infty})}+ \frac{1}{2} S(\omega_{\infty}) I_{E}=
\Big(\frac{d}{Vr}+\frac{\overline{s}}{2}\Big) I_{E},\end{equation} where
$h_{k}=h^{(k)} \otimes g_{\infty}^{\otimes (-k)}$,
$S(\omega_{\infty})$ is the scalar curvature of $\omega_{\infty}$
and $\overline{s}$ is the average of the scalar curvature.
Conversely, if $h_{\infty}$ solves \eqref{eq10}, then there exists a sequence of balanced metrics $h^{(k)}$ on $E \otimes L^k$ for $k \gg 0$ and
$h_{k}\rightarrow h_{\infty}$ in $C^{\infty}$.

\end{theorem}

In the case that the base manifold $X$ has dimension one and the K\"ahler metric $\omega_{\infty}$ has constant curvature, we  prove that the rate of convergence of $h_{k}$ to $h_{\infty}$ is  $O(k^{-\infty})$.

\begin{theorem}\label{thm6}
Let $X$ be a compact Riemann surface and $\omega_{\infty}$ be a
K\"ahler form of constant curvature on $X$. Let $a$ be a positive integer. Suppose that the Hermitian metric $h_{\infty}$ on $E$ satisfies the Hermitian-Einstein equation $$\frac{i}{2\pi}  F_{(E,h_{\infty})}=\omega_{\infty} I_{E}.$$ 
Let $h^{(k)}$ be a sequence of balanced metric on $E \otimes L^k$ for $k \gg 0$ and $h_{k}=h^{(k)} \otimes g_{\infty}^{\otimes (-k)}$. If $h_{k}\rightarrow h_{\infty}$, then $$||h_{k}-h_{\infty}||_{C^a(h_{\infty})}=O(k^{-\infty}).$$ 

\end{theorem}

The proof follows from Theorem \ref{thm5}, lemma \ref{lem4} and lemma \ref{lem5}.

\begin{lemma}\label{lem4}

Let $h$ be a Hermitian metric on $E$. Suppose that $E$ is stable and coefficients  $A_{1},\dots, A_{q}$ in the asymptotic expansion \eqref{eq1} are constant endomorphisms of $E$. If
$q$ is big enough, then there exists a sequence of balanced
metrics $h^{(k)}$ on $E \otimes L^k$ for $k \gg 0$ such that
$$||h-h^{(k)} \otimes g^{\otimes(-k)}||_{C^a(h)}=
O(k^{3+\frac{13n}{2}+\frac{a}{2}-q}).$$

\end{lemma}

\begin{proof}
 First we claim that $$B_{k}(h)=\frac{\chi(k)}{rV}(I_{E}+\sigma_{k}),$$
where $||\sigma_{k}||_{C^{a}}=O(k^{n-q-1}).$ In order to prove this, we observe that there exists a smooth
section  $A(x)$ of $End(E)$ such that
$$B_{k}(h)=
k^n+A_{1}k^{n-1}+....+A_{q}k^{n-q}+A(x)k^{n-q-1}.$$ The bundle $E$
is stable and $A_{j}$'s are constant sections of $End(E)$.
Therefore there exist numbers $a_{1},...,a_{q}$ such that
$A_{j}=a_{j}I_{E}.$ On the other hand
$$\int_{X} tr(B_{k}(h) \frac{\omega_{\infty}^n}{n!}=\chi(k)V,$$
where $V= \int_{X}  \frac{\omega_{\infty}^n}{n!}.$ Thus, $$B_{k}(h)- \frac{\chi(k)}{rV}I_{E}=
\Big(A(x)-\frac{1}{rV}\int_{X} A(x) I_{E}\Big)k^{n-q-1}.$$

Define $\sigma_{k}=\Big(A(x)-\frac{1}{rV}\int_{X} A(x) I_{E}\Big)k^{n-q-1}$, we have
$$B_{k}(h)=\frac{\chi(k)}{rV}(I+\sigma_{k}),$$
where $||\sigma_{k}||_{C^{a}}=O(k^{n-q-1}).$ Now Wang's argument
(\cite [page 276]{W2}) concludes the proof.

\end{proof}

\begin{lemma}\label{lem5}

In the situation of Theorem \ref{thm6}, all coefficients $A_{i}$'s
are constant.

\end{lemma}

\begin{proof}
The coefficients of the asymptotic expansion of the Bergman kernel
are polynomials of the curvature tensor on the base manifold,
curvature tensor on the bundle and their covariant derivatives.
The whole curvature tensors on the base manifold and on the bundle
are constant. Therefore all coefficients are constant.

\end{proof}

\section{Constructing Almost Balanced Metrics }

The goal of this section is to prove Theorem \ref{thm2b}. In order to prove Theorem \ref{thm2b}, we construct  a sequence of almost balanced metrics on $\mathcal{O}_{\mathbb{P}E^*}(d) \otimes L^k$(Theorem \ref{thm7}). We start with definition of balanced metrics on polarized manifolds.

Let $(Y,\omega)$ be a compact K\"ahler manifold of dimension
$n$ and $\mathcal{O}(1)\rightarrow Y$ be a very ample line bundle
on $Y$ equipped with a Hermitian metric $\sigma$ such that
$i\bar{\partial}\partial \log \sigma=\omega$. Since $\mathcal{O}(1)$ is very ample,
using global sections of $\mathcal{O}(1)$, we can embed $Y$ into
$\mathbb{P}(H^0(Y,\mathcal{O}(1))^*)$. A choice of ordered basis
$\underline{s}=(s_{1},...,s_{N})$ of $H^0(Y,\mathcal{O}(1))$ gives
an isomorphism between $\mathbb{P}(H^0(Y,\mathcal{O}(1))^*)$ and
$\mathbb{P}^{N-1}$. Hence for any such $\underline{s}$, we have an
embedding $\iota_{\underline{s}}:Y\hookrightarrow
\mathbb{P}^{N-1}$ such that $\iota_{\underline{s}}^*
\mathcal{O}_{\mathbb{P}^N}(1)=\mathcal{O}(1)$. Using
$\iota_{\underline{s}}$, we can pull back the Fubini-Study metric
and K\"ahler form of the projective space to $\mathcal{O}(1)$ and
$Y$ respectively.

\begin{definition}\label{def2n}

An embedding $\iota_{\underline{s}}$ is called balanced if
$$\int_{Y} \langle  s_{i}, s_{j}
\rangle_{\iota_{\underline{s}}^*h_{\textrm{FS}}}\frac{\iota_{\underline{s}}^*\omega_{\textrm{FS}}^n}{n!}=\frac{V}{N}\delta_{ij},$$
where $V=\int_{Y}\omega^n/n!$. A Hermitian
metric (resp.\ a K\"ahler form) is called balanced if it is
the pull back $\iota^*_{\underline{s}}h_{\textrm{FS}}$ 
(resp.\ $\iota^*_{\underline{s}}\omega_{\textrm{FS}}$) where
$\iota_{\underline{s}}$ is a balanced embedding.

\end{definition}

\begin{remark}
The concepts of balanced metric on holomorphic vector bundles (Definition \ref{def1n}) and balanced metric on polarized manifolds (Definition \ref{def2n}) are different. In order to find a balanced metric on a holomorphic vector bundle $E \rightarrow X$, we need to fix a K\"ahler form $\omega_{0}$ on $X$. A Hermitian metric $h$ on $E$ is balanced (more precisely $\omega_{0}$-balanced) if $B(h, \omega_{0})= C I_{E}$, where $C$ is a constant. But in order to find a balanced metric on a polarized manifold $(Y,\mathcal{O}(1))$, we do not need to fix a K\"ahler form. A positive Hermitian metric $g$ on $\mathcal{O}(1)$ is balanced if $B(g, i\bar{\partial}\partial \log g)$ is constant.

\end{remark}

\begin{definition}

A sequence of Hermitian metrics $h_{k} $ on
$\mathcal{O}(1)\otimes L^{k}$ and ordered bases
$\underline{s}^{(k)}=(s_{1}^{(k)},...,s_{N}^{(k)})$ for
$H^0(Y,\mathcal{O}(1)\otimes L^{k})$ is called \emph{almost
balanced of order $q$} if for any $k$  $$\sum
|s_{i}^{(k)}|_{h_{k}}^2=1$$ and
$$\int_{Y}\langle s_{i}^{(k)}, s_{j}^{(k)} \rangle
_{h_{k}}dvol_{h_{k}}=D^{(k)}\delta_{ij}+M^{(k)}_{ij},$$ where
$D^{(k)}$ is a scalar so that $D^{(k)} \rightarrow C$ as
$k \rightarrow \infty$, where $C$ is a constant  and
$M^{(k)}$ is a trace-free Hermitian matrix such that
$||M^{(k)}||_{\textrm{op}}=O(k^{-q-1})$. Here  $||M^{(k)}||_{op}$ is the operator norm of the matrix $M^{(k)}$.

\end{definition}

For the rest of this section, let  $X$ be a compact Riemann surface and $L$ be an ample line bundle on $X$. Let $g$ be a positive  Hermitian metric on $L$  and $\omega_{\infty}=i \bar{\partial}\partial\log g $ be a K\"ahler form on $X$.
Let $E$ be a holomorphic vector bundle on $X$ of rank $r$ and
slope $\mu$. The slope of $E$ is defined by $\mu=\frac{\deg(E)}{r}$.

Similar to the case of vector spaces, we have the natural
isomorphism $H^{0}(\mathbb{P}E^*,\mathcal{O}_{\mathbb{P}E^*}(d)\otimes L^k)=
H^{0}(X,\textrm{Sym}^d E \otimes L^k)$. Also, any Hermitian metric $H$ on $\textrm{Sym}^d E$ induces a Hermitian
metric $\widehat{H}$ on $\mathcal{O}_{\mathbb{P}E^*}(d) \otimes L^k$.

Suppose that $H$ is a Hermitian metric on $\textrm{Sym}^d E$ and $s_{1},...,s_{N}$ is an orthonormal basis for
$H^{0}(X,\textrm{Sym}^d E \otimes L^k)$ with respect to $L^2(H_{k} ,\omega_{\infty})$, where $H_{k}= H\otimes g^{\otimes k}$. Let $\hat{s_{1}},...,\hat{s_{N}}$ be the
corresponding basis for
$H^{0}(\mathbb{P}E^*,\mathcal{O}_{\mathbb{P}E^*}(d))$. 

We prove that the matrix  $[\int_{\mathbb{P}E^*}
\langle\widehat{s_{i}},\widehat{s_{j}}\rangle_{\widehat{H}}dvol_{\widehat{H}}]$
is close to a scalar matrix. More precisely, we prove the following.

\begin{proposition}\label{prop7}
Let $h_{\infty}$ be a Hermitian-Eienstein metric on $E$, i.e. 
\begin{equation} \label{eq2} iF_{(\bar{\partial}_{E}, h_{\infty})}=\mu \omega_{\infty} I_{ E},\end{equation} where $F_{(\bar{\partial}_{ E}, h_{\infty})}$ is the curvature of the chern connection of  $h_{\infty}$ and $\mu$ is the slope of $E$. Then there exists a constant $C$ depends only on $r$ and $d $ such that if
\begin{equation} \label{eq3}||H-\textrm{Sym}^d h_{\infty}||_{C^2(\textrm{Sym}^d h_{\infty})}\leq \min( \epsilon, \frac{1}{2}),\end{equation}  then
$$\Big|\int_{\mathbb{P}E^*}
\langle\widehat{s_{i}},\widehat{s_{j}}\rangle_{\widehat{H_{k}}}dvol_{\widehat{H_{k}}}-C_{r,d}(d\mu+k)
\delta_{ij}\Big|\leq C\epsilon (d\mu+k).$$ Here $H_{k}= H \otimes g^{\otimes k}$.

\end{proposition}

\begin{proof}

In this proof $C$ denotes a constant depends only on $r$ and $d$ that might change from line to line. 
Define $H_{\infty}=\textrm{Sym}^d h_{\infty}$, $\omega_{0}= i\bar{\partial}\partial\log \widehat{H_{\infty}}$ and $\omega_{k}= \omega_{0}+ k\omega_{\infty}$. Lemma \ref{lem1} implies that $\widehat{H_{\infty}}=\widehat{h_{\infty}}^{\otimes d}$ and therefore $\omega_{0}=d i\bar{\partial}\partial\log \widehat{h_{\infty}}=d \omega_{\widehat{h_{\infty}}} $.
A simple calculation shows that $\omega_{\widehat{h_{\infty}}}^r= r\mu \omega_{\widehat{h_{\infty}}}^{r-1} \wedge \omega_{\infty}$, since $h_{\infty }$ satisfies the Hermitian-Einstein equation \eqref{eq2}. Thus, 
\begin{equation} \label{eq4}\omega_{k}^r= \omega_{0}^r+rk\omega_{0}^{r-1} \wedge \omega_{\infty}=r(d\mu+k) \omega_{0}^{r-1} \wedge \omega_{\infty}.\end{equation} Therefore, 
\begin{align*}\Big|   \int_{\mathbb{P}E^*} \langle\widehat{s_{i}},\widehat{s_{j}}\rangle_{\widehat{H_{k}}} \frac{\omega_{k}^r}{r!}-C_{r,d}(d\mu+k)\delta_{ij} \Big|&=\Big|   (d\mu+k) \int_{\mathbb{P}E^*} \langle\widehat{s_{i}},\widehat{s_{j}}\rangle_{\widehat{H_{k}}} \frac{\omega_{0}^{r-1}}{(r-1)!}\wedge \omega_{\infty}-C_{r,d}(d\mu+k)\delta_{ij} \Big| \\&=(d\mu+k) \Big|    \int_{X}\Big(d^{r-1}\int_{\mathbb{P}E_{x}^*} \langle\widehat{s_{i}},\widehat{s_{j}}\rangle_{\widehat{H_{k}}} \frac{\omega_{\widehat{h_{\infty}}}^{r-1}}{(r-1)!}-C_{r,d}\delta_{ij} \Big)\wedge \omega_{\infty} \Big| \\&\leq  C(d\mu+k) \epsilon \int_{X}  |s_{i}|_{H_{k}} |s_{j}|_{H_{k}} \omega_{\infty}.\end{align*}

The last inequality follows from Proposition \ref{prop1}. Hence  Cauchy-Scwarz inequality implies that 

$$\Big|   \int_{\mathbb{P}E^*} \langle\widehat{s_{i}},\widehat{s_{j}}\rangle_{\widehat{H_{k}}} \frac{\omega_{k}^r}{r!}-C_{r,d}(d\mu+k)\delta_{ij} \Big| \leq C  (d\mu+k) \epsilon \Big(\int_{X}  |s_{i}|^{2}_{H_{k}} \omega_{\infty}\Big)^{\frac{1}{2}}\Big(\int_{X}  |s_{j}|^{2}_{H_{k}} \omega_{\infty}\Big)^{\frac{1}{2}}=C (d\mu+k) \epsilon, $$
since  $s_{1},...,s_{N}$ is an orthonormal basis for $H^{0}(X,\textrm{Sym}^d E \otimes L^k)$ with respect to $L^2(H_{k},\omega_{\infty})$. 

On the other hand,  Lemma \ref{lem7} implies that $||\omega-\omega_{0}||_{C^0(\omega_{0})} \leq C\epsilon $. Therefore,  
$$||(\omega+ k\omega_{\infty})-\omega_{k}||_{C^0(\omega_{k})}=||(\omega+ k\omega_{\infty})-(\omega_{0}+k\omega_{\infty})||_{C^0(\omega_{k})} \leq ||\omega-\omega_{0}||_{C^0(\omega_{0})} \leq C\epsilon,$$  since $\omega_{\infty}$ is a semipositive $(1,1)$-form on $\mathbb{P}E^*$.
Applying Lemma \ref{lem8} implies that \begin{equation} \label{eq5} \Big| dvol_{\widehat{H_{k}}}-   \frac{\omega_{k}^r}{r!}  \Big| \leq C\epsilon \frac{\omega_{k}^r}{r!}.\end{equation} 
Thus,  \begin{align*}\Big|\int_{\mathbb{P}E^*} \langle\widehat{s_{i}},\widehat{s_{j}}\rangle_{\widehat{H_{k}}}dvol_{\widehat{H_{k}}}-\int_{\mathbb{P}E^*} \langle\widehat{s_{i}},\widehat{s_{j}}\rangle_{\widehat{H_{k}}}\frac{\omega_{k}^r}{r!}\Big |&\leq \int_{\mathbb{P}E^*}| \langle\widehat{s_{i}},\widehat{s_{j}}\rangle_{\widehat{H_{k}}}| |dvol_{\widehat{H_{k}}}-\frac{\omega_{k}^r}{r!}| \\&\leq C\epsilon \int_{\mathbb{P}E^*}  |s_{i}|_{H_{k}} |s_{j}|_{H_{k}}  \frac{\omega_{k}^r}{r!}\\& \leq C(d\mu+k)\epsilon \int_{\mathbb{P}E^*}|s_{i}|_{H_{k}} |s_{j}|_{H_{k}}  \frac{\omega_{0}^{r-1}}{(r-1)!}\wedge \omega_{\infty}\\&=C(d\mu+k)\epsilon \int_{X}|s_{i}|_{H_{k}} |s_{j}|_{H_{k}}  \omega_{\infty} \\&\leq C\epsilon  \Big(\int_{X}  |s_{i}|^{2}_{H_{k}} \omega_{\infty}\Big)^{\frac{1}{2}}\Big(\int_{X}  |s_{j}|^{2}_{H_{k}} \omega_{\infty}\Big)^{\frac{1}{2}} \leq C\epsilon. \end{align*}

Here we used \eqref{eq4}, \eqref{eq5} and the fact $\sup_{\mathbb{P}E_{x}^*}|\widehat{s_{i}}|_{\widehat{H_{k}}} = |s_{i}(x)|_{H_{k}}$.  We have

\begin{align*}\Big|&\int_{\mathbb{P}E^*}
\langle\widehat{s_{i}},\widehat{s_{j}}\rangle_{\widehat{H_{k}}}dvol_{\widehat{H_{k}}}- C_{r,d}(d\mu+k)
\delta_{ij}\Big| \\&\leq \int_{\mathbb{P}E^*}| \langle\widehat{s_{i}},\widehat{s_{j}}\rangle_{\widehat{H_{k}}}| |dvol_{\widehat{H_{k}}}-\frac{\omega_{k}^r}{r!}|+ \Big|    \int_{\mathbb{P}E^*} \langle\widehat{s_{i}},\widehat{s_{j}}\rangle_{\widehat{H_{k}}} \frac{\omega_{k}^r}{r!}-C_{r,d}(d\mu+k)\delta_{ij} \Big|\\& \leq C(d\mu+k)\epsilon.\end{align*}


\end{proof}


\begin{theorem}\label{thm7}
Let $X$ be a compact Riemann surface and $L \rightarrow X$ be an ample line bundle equipped with a Hermitian metric $g$. Suppose that  $i\bar{\partial}\partial\log g =\omega_{\infty}$ is a
K\"ahler form on $X$. Let $E$ be a stable holomorphic vector bundle of rank $r$ on $X$ and $h_{\infty}$ is a Hermitian-Einstein metric on $E$. 
Let $R$ be the rank of $\textrm{Sym}^d E$. Suppose that $\{s_{i}^{(k)}\}_{i=1}^{N_{k}}$ is a sequence of bases for $H^0(X,\textrm{Sym}^d E
\otimes L^{\otimes k}) $ and $H^{(k)}$ is a sequence of Hermitian metrics on $\textrm{Sym}^d E
\otimes L^{\otimes k}$ such that
\begin{align*}&\sum_{i=1}^{N_{k}} s_{i}^{(k)} \otimes
(s_{i}^{(k)})^{*_{H^{(k)}}}=I_{\textrm{Sym}^d E
\otimes L^k},\\&\int_{X} \langle
s_{i}^{(k)},s_{j}^{(k)}\rangle_{H^{(k)}} \omega_{\infty}
=\frac{RVol(X,\omega_{\infty})}{N_{k}} \delta_{ij}, \\& ||H^{(k)} \otimes g^{\otimes (-k)}-\textrm{Sym}^d h_{\infty}||_{C^2(\textrm{Sym}^d h_{\infty})}=O(k^{-\infty}).\end{align*}
Then the sequence of Hermitian metrics $H^{(k)}$ on $\mathcal{O}_{\mathbb{P}E^*}(d)\otimes L^k$ and ordered bases $\underline{s}^{(k)}=
( \widehat{s_{1}^{(k)}},\dots , \widehat{s_{N_{k}}^{(k)}})$ of $H^{0}(\mathbb{P}E^*,\mathcal{O}_{\mathbb{P}E^*}(d) \otimes L^k)$ is almost balanced of order $q$ for any positive integer $q$.

\end{theorem}

\begin{proof}

Let $p$ be a positive integer. There exists a constant $C$ independent of $k$ such that $$||H^{(k)} \otimes g^{\otimes (-k)}-\textrm{Sym}^d h_{\infty}||_{C^2(\textrm{Sym}^d h_{\infty})} \leq Ck^{-p}.$$ Fix $k \gg 0$. The basis $\{\sqrt{R^{-1}N_{k}}s_{1}^{(k)}, \dots , \sqrt{R^{-1}N_{k}}s_{N_{k}}^{(k)}\}$ is an orthonormal basis for $H^{0}(X,\textrm{Sym}^d E \otimes L^k)$ with respect to $L^2(H^{(k)} ,\omega_{\infty})$. Applying Proposition \ref{prop7} to $H=H^{(k)}\otimes g^{\otimes -k}$ implies that 
\begin{equation} \label{eq6}\Big|\int_{\mathbb{P}E^*}\langle\widehat{s_{i}^{(k)}},\widehat{s_{j}^{(k)}}\rangle_{\widehat{H^{(k)}}}dvol_{\widehat{H^{(k)}}}-C_{r,d}(d\mu+k)\delta_{ij}\Big|\leq Ck^{-p} (d\mu+k).\end{equation} 
Define \begin{align*}& M^{(k)}= \int_{\mathbb{P}E^*}\langle\widehat{s_{i}^{(k)}},\widehat{s_{j}^{(k)}}\rangle_{\widehat{H^{(k)}}}dvol_{\widehat{H^{(k)}}}-C_{r,d}(d\mu+k)\delta_{ij}, \\& D^{(k)}= C_{r,d} (d\mu+k ).  \end{align*}
We have $$\int_{\mathbb{P}E^*}\langle\widehat{s_{i}^{(k)}},\widehat{s_{j}^{(k)}}\rangle_{\widehat{H^{(k)}}}dvol_{\widehat{H^{(k)}}}=D^{(k)}I+M^{(k)}.$$
A simple calculation shows that $$ D^{(k)}\rightarrow C_{r,d}\,\,\,\,\,\,\, \textrm{as}\,\,\,\,
k\rightarrow \infty.$$
On the other hand, \eqref{eq6} implies  \begin{align*} ||M^{(k)}||_{op} & \leq \sum_{ij} |(M^{(k)})_{ij}| \leq Ck^{-p}(d\mu+k)N_{k}^2 \leq C^{\prime}k^{3-p}.     \end{align*} Note that $N_{k}=O(k)$ by Riemann-Roch theorem.
Therefore for any positive integer $q$, $||M^{(k)}||=O(k^{-q-1})$ which means that the sequence of Hermitian metrics $H^{(k)}$ on $\mathcal{O}_{\mathbb{E}^*}(d)\otimes L^k$ and ordered bases $\underline{s}^{(k)}=( \widehat{s_{1}^{(k)}},\dots , \widehat{s_{N_{k}}^{(k)}})$ of $H^{0}(\mathbb{P}E^*,\mathcal{O}_{\mathbb{P}E^*}(d) \otimes L^k)$ is almost balanced of order $q$ for any positive integer $q$.

\end{proof}

\begin{proof} [Proof of Theorem \ref{thm2b}.]

Fix a positive integer $a \geq 4$. Let $\omega_{\infty } $ be the k\"ahler form on $X$ with constant curvature. Since $E$ is a stable bundle, there exists a Hermitian metric $h_{\infty}$ on $E$ satisfies the Hermitian-Einstein equation \eqref{eq2}.
Therefore Theorem \ref{thm5} and Theorem \ref{thm6} imply that there exists a sequence of balanced metrics $H^{(k)}$ on 
$\textrm{Sym}^d E \otimes L^k$ such that
\begin{equation} \label{eq7}||H^{(k)} \otimes g^{\otimes (-k)}-\textrm{Sym}^d h_{\infty}||_{C^a(\textrm{Sym}^d h_{\infty})}=O(k^{-\infty}).\end{equation} 
By definition of balanced metrics on vector bundles (Definition \ref{def1n}), there exists a sequence of bases  $\{s_{i}^{(k)}\}_{i=1}^{N_{k}}$  for $H^0(X,\textrm{Sym}^d E
\otimes L^{\otimes k}) $ such that
\begin{align*}&\sum_{i=1}^{N_{k}} s_{i}^{(k)} \otimes (s_{i}^{(k)})^{*_{H^{(k)}}}=I_{\textrm{Sym}^d E \otimes L^k},\\&\int_{X} \langle
s_{i}^{(k)},s_{j}^{(k)}\rangle_{H^{(k)}} \omega_{\infty}=\frac{RVol(X,\omega_{\infty})}{N_{k}} \delta_{ij},\end{align*} where $R$ is the rank of $\textrm{Sym}^d E$. Hence \begin{equation} \label{eq8}\sum_{i=1}^{N_{k}} |\widehat{ s_{i}^{(k)} }|^2_{\widehat{H^{(k)}}}=1.\end{equation} 
Define $\omega_{0}= i\bar{\partial}\partial\log \widehat{H_{\infty}} $ and $\widetilde{\omega_{k}}= i\bar{\partial}\partial\log \widehat{H^{(k)}}$. Thus \eqref{eq7} implies
$$||\widetilde{\omega_{k}}-\omega_{k}||_{C^{a}(\omega_{k})}  \leq ||(\widetilde{\omega_{k}}-k\omega_{\infty})-\omega_{0}||_{C^{a}(\omega_{0})}=O(k^{-\infty}),$$
$$|\log \widehat{H^{(k)}}- \log (\widehat{H_{\infty}} \otimes g^{\otimes k})|_{C^{a+2}}= |\log (\widehat{H^{(k)}} \otimes g^{\otimes (-k)})- \log \widehat{H_{\infty}} |_{C^{a+2}}=O(k^{-\infty}). $$

On the other hand, Theorem \ref{thm7} and \eqref{eq8} imply that the sequence of Hermitian metrics $\widehat{H^{(k)}}$  on $\mathcal{O}_{\mathbb{P}E^*}(d)\otimes L^k$ and ordered bases $\underline{s}^{(k)}=( \widehat{s_{1}^{(k)}}, \dots  \widehat{s_{N_{k}}^{(k)}})$ for $H^{0}(\mathbb{P}E^*,\mathcal{O}_{\mathbb{P}E^*}(d)\otimes L^k)$ is almost balanced of order $q$ for any positive integer $q$. Since $\mathbb{P}E^*$ has no nontrivial
holomorphic vector fields, we can perturb these almost balanced metrics to get balanced metrics on $\mathcal{O}_{\mathbb{P}E^*}(d)\otimes \pi^* L^k$ for $k \gg 0$ (see \cite[Theorem 4.6]{S}).

\end{proof}

\end{document}